\documentclass[
submission
]{dmtcs-episciences}

\usepackage[utf8]{inputenc}
\usepackage{subfigure}

\newtheorem{theorem}{Theorem}[section]
\newtheorem{lemma}[theorem]{Lemma}

\newtheorem{proposition}[theorem]{Proposition}
\newtheorem{definition}{Definition}

\usepackage[round]{natbib}

\author{Huazhong L\"{u}\affiliationmark{1}\thanks{Corresponding author.}
  \and Tingzeng Wu\affiliationmark{2}\thanks{This research was partially supported by the National Natural Science Foundation of China (Nos. 11801061 and 11761056), the Chunhui Project of Ministry of Education (No. Z2017047) and the Fundamental Research Funds for the Central Universities (No. ZYGX2018J083)}}
\title[Super edge-connectivity and matching preclusion of data center networks]{Super edge-connectivity and matching preclusion of data center networks}
\affiliation{
  School of Mathematical Sciences, University of Electronic Science and Technology of China, Chengdu, Sichuan 610054, P.R. China\\
  School of Mathematics and Statistics, Qinghai Nationalities University, Xining, Qinghai 810007, P.R. China}
\received{2018-7-17}
\revised{2019-2-22}
\accepted{2019-6-28}
\begin{document}
\publicationdetails{21}{2019}{4}{2}{4689}
\maketitle
\begin{abstract}
Edge-connectivity is a classic measure for reliability of a network in the presence of edge failures. $k$-restricted edge-connectivity is one of the refined indicators for fault tolerance of large networks. Matching preclusion and conditional matching preclusion are two important measures for the robustness of networks in edge fault scenario. In this paper, we show that the DCell network $D_{k,n}$ is super-$\lambda$ for $k\geq2$ and $n\geq2$, super-$\lambda_2$ for $k\geq3$ and $n\geq2$, or $k=2$ and $n=2$, and super-$\lambda_3$ for $k\geq4$ and $n\geq3$. Moreover, as an application of $k$-restricted edge-connectivity, we study the matching preclusion number and conditional matching preclusion number, and characterize the corresponding optimal solutions of $D_{k,n}$. In particular, we have shown that $D_{1,n}$ is isomorphic to the $(n,k)$-star graph $S_{n+1,2}$ for $n\geq2$.
\end{abstract}

%
%

\section{Introduction}
\label{sec:in}

Let $G=(V(G),E(G))$ be a graph, where $V(G)$ is the vertex-set of $G$ and $E(G)$ is the edge-set of $G$. The number of vertices of $G$ is denoted by $|V(G)|$. The degree of a vertex $u$ in $G$ is denoted by $d_G(u)$. For any $X\subset V(G)$, we use $G[X]$ to denote the subgraph of $G$ induced by $X$. For other standard graph notations not defined here please refer to \cite{Bondy}.

Networks are usually modeled as graphs, and the edge-connectivity is a classic measurement for the fault tolerance of the graph. In general, the larger the edge-connectivity of the graphs, the higher the reliability of the corresponding networks. It is well-known that $\lambda(G)\leq\delta(G)$, where $\lambda(G)$ and $\delta(G)$ are the edge-connectivity and the minimum degree of $G$, respectively. To precisely measure the reliability of graphs, \cite{Esfahanian} introduced a more refined index, namely the restricted edge-connectivity. Later, \cite{Fabrega} introduced the $k$-restricted edge-connectivity as a generalisation of this concept.

An edge-cut $F$ is called a $k$-restricted edge-cut if every component of $G-F$ contains at least $k$ vertices ($k\geq2$). The {\em $k$-restricted edge-connectivity} $\lambda_k(G)$, if exists, is the minimum cardinality over all $k$-restricted edge-cuts in $G$. Let $X$ be a vertex subset of $G$ and let $\overline{X}$ be the complement of $X$, namely $\overline{X}=V(G)\setminus X$. We denote the edges between $X$ and $\overline{X}$ by $[X,\overline{X}]$. The {\em minimum $k$-edge degree} of a graph $G$ for integers $k\geq2$, is $\xi_k(G)=\min\{|[X,\overline{X}]|:|X|=k$ and $G[X]$ is connected$\}$.

For a graph $G$ satisfying $\lambda_k(G)\leq \xi_k(G)$, if $\lambda_k(G)= \xi_k(G)$ holds, then it is called {\em $\lambda_k$-optimal}. In particular, $\lambda_2$ is the restricted edge-connectivity, and accordingly $\xi_2$ is known as the edge degree.

For $\lambda_2(G)$, \cite{Esfahanian} showed that each connected graph $G$ of order at least 4 except a star ($K_{1,n-1}$) has a restricted edge-cut and satisfies $\lambda(G)\leq\lambda_2(G)\leq\xi_2(G)$. Moreover, \cite{Bonsma} have shown that if $\lambda_3(G)$ exists, then $\lambda_3(G)\leq\xi_3(G)$.

A graph $G$ is super-$\lambda$ (resp. super-$\lambda_k$) if each minimum edge-cut (resp. $k$-restricted edge-cut) isolates a singleton (resp. a connected subgraph of order $k$). It is obvious that if $G$ is super-$\lambda_k$, then $G$ is $\lambda_k$-optimal, whereas the reverse does not hold. Generally, a graph is {\em super $m$-edge-connected of order $q$} if when at least $m$ edges deleted, the resulting graph is either connected or it has one big component and a number of small components with at most $q$ vertices in total. Obviously, a super-$\lambda$ graph is super $\lambda(G)$-edge-connected of order 1.

A {\em perfect matching} of a graph $G$ is an independent edge set that saturates all vertices of $G$. For an edge subset $F$ of an graph $G$ with even order, if $G-F$ has no perfect matching in $G$, then $F$ is called a {\em matching preclusion set} of $G$. The {\em matching preclusion number}, denoted by $mp(G)$, is defined to be the minimum cardinality among all matching preclusion sets. Any such set of size $mp(G)$ is called an {\em optimal matching preclusion set} (or optimal solution). This concept was proposed by \cite{Brigham} as a measure of robustness of networks, as well as a theoretical connection with conditional connectivity and ``changing and unchanging of invariants''. Therefore, networks of larger $mp(G)$ signify higher fault tolerance under edge failure assumption.

It is obvious that the edges incident to a common vertex form a matching preclusion set. Any such set is called a {\em trivial solution}. Therefore, $mp(G)$ is no greater than $\delta(G)$. A graph is {\em super matched} if $mp(G)=\delta(G)$ and each optimal solution is trivial. In the random link failure scenario, the possibility of simultaneous failure of links in a trivial solution is very small. Motivated by this, \cite{Cheng0} introduced the following definition to seek obstruction sets excluding those induced by a single vertex. The {\em conditional matching preclusion number} of $G$, denoted by $mp_1(G)$, is the minimum number of edges whose deletion results in the graph with neither a perfect matching nor an isolated vertex. If the resulting graph has no isolated vertices after edge deletion, a path $u\rightarrow v\rightarrow w$, where the degree of both $u$ and $w$ are 1, is a basic obstruction to perfect matchings. So to generate such an obstruction set, one can pick any path $u\rightarrow v\rightarrow w$ in the original graph, and delete all the edges incident to $u$ and $w$ but not $v$. We define $v_e(G)=\min\{d_G(u)+d_G(w)-2-y_G(u,w):$ $u$ and $v$ are ends of a path of length 2$\}$, where $y_G(u,w)=1$ if $uw\in E(G)$ and 0 otherwise.

\begin{proposition}\cite{Cheng0} For a graph $G$ of even order and $\delta(G)\geq3$, mp$_1(G)\leq v_e(G)$ holds.
\end{proposition}

A conditional matching preclusion set of $G$ that achieves $mp_1(G)=v_e(G)$, a set of edges whose removal leaves the subgraph without perfect matchings and with no isolated vertices, is called an {\em optimal conditional matching preclusion set} (or optimal conditional solution). An optimal conditional solution of the basic form induced by a 2-path giving $v_e(G)$ is a {\em trivial optimal conditional solution}. As mentioned earlier, the matching preclusion number measures the robustness of the requirement in the link failure scenario, so it is desirable for an interconnection network to be super matched. Analogously, it is desirable to have the property that all the optimal conditional solutions are trivial as well. The interconnection network possesses the above property is called {\em conditionally super matched}.

Until now, the matching preclusion number of numerous networks were calculated and the corresponding optimal solutions were obtained, such as the complete graph, the complete bipartite graph and the hypercube (\cite{Brigham}), Cayley graphs generated by 2-trees and hyper Petersen networks (\cite{Cheng}), Cayley graphs generalized by transpositions and $(n,k)$-star graphs (\cite{Cheng2}), restricted HL-graphs and recursive circulant $G(2^m,4)$ (\cite{Park}), tori and related Cartesian products (\cite{Cheng3}), $(n,k)$-bubble-sort graphs (\cite{Cheng4}), balanced hypercubes (\cite{Lu}), burnt pancake graphs (\cite{Hu}), $k$-ary $n$-cubes (\cite{Wang.S2}), cube-connected cycles (\cite{Li}), vertex-transitive graphs (\cite{Li0}), $n$-dimensional torus (\cite{Hu2}), binary de Bruijn graphs (\cite{Lin}) and $n$-grid graphs (\cite{Ding}). For the conditional matching preclusion problem, it is solved for the complete graph, the complete bipartite graph and the hypercube (\cite{Cheng0}), arrangement graphs (\cite{Cheng5}), alternating group graphs and split-stars (\cite{Cheng6}), Cayley graphs generated by 2-trees and the hyper Petersen networks (\cite{Cheng}), Cayley graphs generalized by transpositions and $(n,k)$-star graphs (\cite{Cheng2}), burnt pancake graphs (\cite{Cheng00,Hu}), balanced hypercubes (\cite{Lu}), restricted HL-graphs and recursive circulant $G(2^m,4)$ (\cite{Park}), $k$-ary $n$-cubes (\cite{Wang.S2}), hypercube-like graphs (\cite{Park2}) and cube-connected cycles (\cite{Li}). Particularly, \cite{Lu1} has proved recently that it is NP-complete to determine the matching preclusion number and conditional matching preclusion number of a connected bipartite graph.

Data centers are crucial to the business of companies such as Amazon, Google and Microsoft. Data centers with large number of servers were built to offer desirable on-line applications such as web search, email, cloud storage, on-line gaming, etc. Data center networks $D_{k,n}$, DCell in short, was introduced by \cite{Guo} for parallel computing systems, which has numerous desirable features for data center networking. In DCell, a large number of servers are connected by high-speed links and switches, providing much higher network capacity compared with the tree-based systems. Several attractive properties of DCell has been explored recently, such as Hamilton property (\cite{Wang.X0}), pessimistic diagnosability (\cite{Gu}), the restricted $h$-connectivity (\cite{Wang.X}), vertex-transitivity (\cite{Lu2}) and disjoint path covers (\cite{Wang.X1}).

The restricted edge-connectivity and extra (edge) connectivity of lots of famous networks were studied  by \cite{Bonsma,Chang,Fabrega,Lu3,Zhu}. \cite{Wang.X} obtained the restricted $h$-connectivity of the DCell, which is the connectivity of $G$ under the restriction that each fault-free vertex has at least $h$ fault-free neighbors in $G$. In the same paper, the authors proposed an interesting problem that whether similar results of restricted edge-connectivity apply to the DCell network. In this paper, we study this problem and show that the DCell network $D_{k,n}$ is super-$\lambda$ for $k\geq2$ and $n\geq2$, super-$\lambda_2$ for $k\geq3$ and $n\geq2$ or $k=2$ and $n=2$, and super-$\lambda_3$ for $k\geq4$ and $n\geq3$. As a direct application of the above result, we obtain the matching preclusion number and conditional matching preclusion number, and characterize the corresponding optimal solutions of the DCell.

The rest of this paper is organized as follows. The definition of the DCell and some notations are given in Section \ref{sec:pre}. The restricted edge-connectivity of the DCell is computed in Section \ref{sec:super}. The (conditional) matching preclusion number of the DCell is obtained in Section \ref{sec:matching}.

\section{Preliminaries}
\label{sec:pre}

We begin with the definition of the DCell.

\begin{definition}\label{def}\cite{Wang.X}
A $k$ level DCell for each $k$ and some global constant $n$, denoted by $D_{k,n}$, is recursively defined as follows. Let $D_{0,n}$ be the complete graph $K_n$ and let $t_{k,n}$ be the number of vertices in $D_{k,n}$. For $k\geq1$, $D_{k,n}$ is constructed from $t_{k-1,n}+1$ disjoint copies of $D_{k-1,n}$, where $D_{k-1,n}^i$ denotes the $i$th copy. Each pair of $D_{k-1,n}^a$ and $D_{k-1,n}^b$ ($a<b$) is joined by a unique $k$ level edge below.

A vertex of $D_{k-1,n}^i$ is labeled by $(i,a_{k-1},\cdots,a_0)$, where $k\geq1$ and $a_0\in\{0,1,\cdots,$ $n-1\}$. The suffix $(a_{j},a_{j-1},\cdots,a_0)$, of a vertex $v$, has the unique $uid_j$, given by $uid_j(v)=a_0+\sum_{l=1}^{j}(a_lt_{l-1,n})$. The vertex $uid_{k-1}$ $b-1$ of $D_{k-1,n}^a$ is connected to $uid_{k-1}$ $a$ of $D_{k-1,n}^b$.
\end{definition}

By the definition above, it is obvious that $D_{0,n}$ is the complete graph $K_n$ ($n\geq2$) and $D_{1,2}$ is a 6-cycle. We illustrate some $D_{k,n}$ with small parameters $k$ and $n$ in Fig. \ref{DCell}. By Definition \ref{def}, we know that there exists exactly one edge, called a {\em level} $k$ edge, between $D_{k-1,n}^a$ and $D_{k-1,n}^b$. For convenience, let $E_k$ denote the set of all level $k$ edges of $D_{k,n}$. Let $F\subseteq E(D_{k,n})$ and $p=|V(D_{k-1,n})|$, we denote $F^i=E(D_{k-1,n}^i)\cap F$ and $f^i=|F^i|$ for $0\leq i\leq p$. We use $e_k(u)$ to denote the level $k$ edge incident with $u$ and $u^k$ to denote its level $k$ neighbor.

\begin{figure}[htbp]
\centering
\includegraphics[height=75mm]{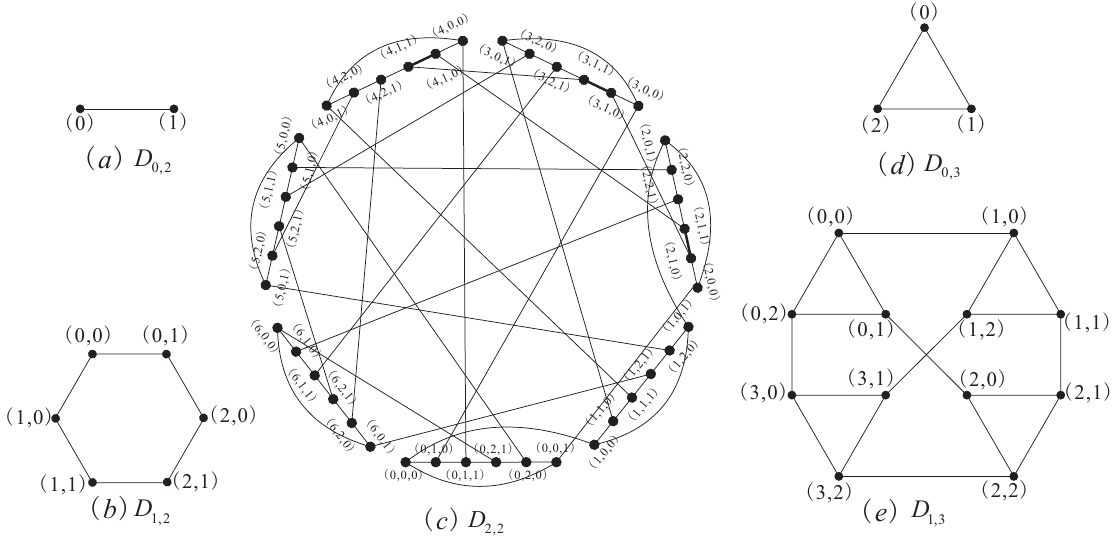}
\caption{Some small DCells.} \label{DCell}
\end{figure}

\section{Super edge-connectivity of DCell}
\label{sec:super}

It is not hard to see that $\lambda(D_{k,n})=n+k-1$. Observe that the edges coming from a complete subgraph $K_n$ form a non-trivial minimum edge-cut of $D_{1,n}$, so $D_{1,n}$ is not super-$\lambda$ for $n\geq2$. For $k\geq2$ and $n\geq2$, we have the following result.

\begin{theorem}\label{super-lambda} $D_{k,n}$ is super-$\lambda$ for all $k\geq2$ and $n\geq2$.
\end{theorem}
\begin{proof} By Definition \ref{def}, $D_{k,n}$ can be split into $p+1$ copies of $D_{k-1,n}$, denoted by $D_{k-1,n}^i$, $0\leq i\leq p$. It is clear that every vertex in $D_{k-1,n}^i$ has exactly one neighbor not in $D_{k-1,n}^i$. In addition, there is exactly one edge between $D_{k-1,n}^i$ and $D_{k-1,n}^j$ for $i\neq j$. Let $F$ be any minimum edge-cut of $D_{k,n}$, then $|F|=n+k-1$. Assume that $D_{k,n}-F$ is disconnected. We need to show that $F$ is the set of edges incident to a unique vertex.

\begin{enumerate}[{Case }1:]
\item $f^i\leq n+k-3$ for each $0\leq i\leq p$. Obviously, $D_{k-1,n}^i-F^i$ is connected since $D_{k-1,n}^i$ is $n+k-2$ edge-connected. By contracting each $D_{k-1,n}^i$ of $D_{k,n}$ into a singleton, we obtain a complete graph $K_{p+1}$. Moreover, the edges of $K_{p+1}$ obtained above correspond to all level $k$ edges in $D_{k,n}$. It is clear that $p>n+k-1$ whenever $n\geq2$ and $k\geq2$, therefore, $K_{p+1}$ is connected when we delete at most $n+k-1$ edges. (This fact will be used time and time again in the remainder of this paper.) This implies that $D_{k,n}-F$ is connected, a contradiction.

\item $f^i\geq n+k-2$ for some $i\in\{0,\cdots, p\}$. Suppose without loss of generality that $f^0\geq n+k-2$. If $f^0=n+k-1$, then $F^i=\emptyset$ for $1\leq i\leq p$. Since each vertex in $D_{k-1,n}^0$ has exactly one neighbor not in $D_{k-1,n}^0$, $D_{k,n}-F$ is connected, a contradiction. We now assume that $f^0=n+k-2$. If $D_{k-1,n}^0-F^0$ is connected, by the discussion in Case 1, then $D_{k,n}-F$ is connected. So we assume that $D_{k-1,n}^0-F^0$ is disconnected and $C$ is one of its components. Clearly, $D_{k,n}-V(D_{k-1,n}^0)-F$ is connected. If $C$ is a singleton and, furthermore, the level $k$ edge incident to $C$ is contained in $F$, then $F$ is a super edge-cut of $D_{k,n}$; otherwise $D_{k,n}-F$ is connected. If $C$ consists of at least two vertices, noting each vertex of $D_{k-1,n}^0$ has a neighbor not in $D_{k-1,n}^0$, then $C$ is connected to $D_{k,n}-V(D_{k-1,n}^0)-F$, yielding that $D_{k,n}-F$ is connected, a contradiction. Hence, the statement holds.
\end{enumerate}
\end{proof}

As mentioned earlier, there exists a non-trivial restricted edge-cut if $k=1$, which implies that $D_{1,n}$ is not super-$\lambda_2$ for all $n\geq2$.

\begin{lemma}\label{Complete-super-lambda-dash}\cite{Wang.S} The complete graph $K_n$ is super-$\lambda_2$ for $n\geq 4$.
\end{lemma}

\begin{lemma}\label{D-k-n-common-neighbor}  Let $uv$ be any edge in $D_{k,n}$ for $n\geq2$ and $k\geq1$. If $uv$ is a level 0 edge, then $u$ and $v$ have exactly $n-2$ common neighbors; if $uv$ is a level $j$ edge, $1\leq j\leq k$, then $u$ and $v$ have no common neighbors.
\end{lemma}
\begin{proof} If $uv$ is a level 0 edge, then $uv$ lies in a complete subgraph $K_n$ ($D_{0,n}$) of $D_{k,n}$. Clearly, $u$ and $v$ have exactly $n-2$ common neighbors in this $K_n$. If $u$ and $v$ have another common neighbor $w$ outside this $K_n$, then a triangle $uvwu$ occurs, which is impossible according to Definition \ref{def}. Similarly, if $uv$ is a level $k$ edge, then $u$ and $v$ have no common neighbors. This completes the proof.
\end{proof}

\begin{theorem}\label{D-k-n-lambda-dash} $\lambda_2(D_{k,n})=2n+2k-4$ for all $n\geq2$ and $k\geq2$.
\end{theorem}
\begin{proof} Since $D_{k,n}$ is ($n+k-1$)-regular, we have $\xi_2(D_{k,n})=2n+2k-4$. Additionally, $D_{k,n}$ is not a star for $n\geq2$ and $k\geq2$, then $\lambda_2(D_{k,n})\leq\xi_2(D_{k,n})=2n+2k-4$. We only need to show that $\lambda_2(D_{k,n})\geq2n+2k-4$.

Let $F$ be any subset of edges in $D_{k,n}$ such that $|F|\leq2n+2k-5$ and there is no isolated vertex in $D_{k,n}-F$. We shall prove that $D_{k,n}-F$ is connected. We may assume that $|F|=2n+2k-5$. Suppose without loss of generality that $f^0$ is the largest one among $f^i$. Then $f^j\leq n+k-3$ for each $j\in\{1,\cdots,p\}$. Since $\lambda(D_{k-1,n})=n+k-2$, each of $D_{k-1,n}^j-F^j$ is connected. By contracting each $D_{k-1,n}^i$ into a singleton, we obtain a complete graph $K_{p+1}$. Note that $p\geq n(n+1)\cdots(n+k-1)$, we have $p>2n+2k-5$ whenever $n\geq2$ and $k\geq2$, which implies that $D_{k,n}-V(D_{k-1,n}^0)-F$ is connected. It remains to show that any vertex in $D_{k-1,n}^0-F^0$ is connected to a vertex in $D_{k,n}-V(D_{k-1,n}^0)-F$ via a fault-free path. If $D_{k-1,n}^0-F^0$ is connected, then $D_{k,n}-F$ is connected. We assume that $D_{k,n}-F$ is disconnected. Thus, $f^0\geq n+k-2$.

Suppose that $u$ is an arbitrary vertex in $D_{k-1,n}^0$. If $uu^k\not\in F$, we are done. So we assume that $uu^k\in F$. Since there exists no isolated vertex in $D_{k,n}-F$, there exist an edge $uv$ incident with $u$ ($v\neq u^k$) in $D_{k-1,n}^0$ such that $uv\not\in F^0$. Moreover, if $vv^k\not\in F$, we are done. So we assume that $vv^k\in F$. We consider the following two cases.

\begin{enumerate}[{Case }1:]
\item $uv$ is a level $l$ edge, $1\leq l\leq k-1$. By Lemma \ref{D-k-n-common-neighbor}, $u$ and $v$ have no common neighbors. Let $E_1=\{uu_l:uu_l\in E(D_{k-1,n}^0)\setminus\{uv\}\}$ and $E_2=\{vv_l:vv_l\in E(D_{k-1,n}^0)\setminus\{uv\}\}$, then $|E_1|=|E_2|=n+k-3$. It is not difficult to see that there are $n+k-3$ edge disjoint paths from $u$ (resp. $v$) to a vertex in $D_{k,n}-V(D_{k-1,n}^0)-F$ via $u_l$ (resp. $v_l$). Observe that $2(n+k-3)+2>2n+2k-5$, so $u$ is connected to a vertex in $D_{k,n}-V(D_{k-1,n}^0)-F$ via a fault-free path.

\item $uv$ is a level 0 edge. By Lemma \ref{D-k-n-common-neighbor}, $u$ and $v$ have exactly $n-2$ common neighbors $w_1,w_2,\cdots,$ $w_{n-2}$ in $D_{k-1,n}^0$. In fact, $w_1,w_2,\cdots,w_{n-2}$ are in a complete subgraph $K_n$ of $D_{k-1,n}^0$. Besides, $u$ (resp. $v$) has $k-1$ distinct neighbors $u_j$ (resp. $v_j$) outside $K_n$, $1\leq j\leq k-1$. So there exist $k-1$ edge disjoint paths $P_j=uu_ju_j^k$ (resp. $Q_j=vv_jv_j^k$) from $u$ (resp. $v$) to $u_j^k$ (resp. $v_j^k$), where $u_j^k$ (resp. $v_j^k$) is the level $k$ neighbor of $u_j$ (resp. $v_j$). If at least one of $P_j$ and $Q_j$ is fault-free in $D_{k,n}-F$, we are done. So we assume that each of $P_j$ and $Q_j$ has at least one edge in $F$.

There are at most $2n+2k-5-2\times(k-1)-2=2n-5$ edges of $F$ in the $K_n$ of $D_{k-1,n}^0$ that contains $uv$. Clearly, we only need to consider $n\geq3$ since $2n-5<0$ when $n=2$. Clearly, $K_n-F^0$ is connected since $2n-5=1$ when $n=3$. In addition, by Lemma \ref{Complete-super-lambda-dash}, $K_n$ is super-$\lambda_2$ when $n\geq4$. In other words, $\lambda_2(K_n)=2n-4$ when $n\geq4$. If $K_n-F^0$ is connected, for each $w_l$, $1\leq l\leq n-2$, there are $k-1$ distinct neighbors not in $D_{k-1,n}^0$ and exactly one level $k$ neighbor. Since $k(n-2)>2n-5$ whenever $n\geq3$ and $k\geq2$, there exists a fault-free path from $u$ to a vertex in $D_{k,n}-V(D_{k-1,n}^0)-F$. If $K_n-F^0$ is disconnected, it follows that $n\geq4$ and there exists a singleton, say $w_{n-2}$, in $K_n-F^0$. Then there are $k(n-3)$ edge disjoint paths from the large component of $K_n-F^0$ to $D_{k,n}-V(D_{k-1,n}^0)-F$. Since $k(n-3)>2n-5-(n-1)$ whenever $n\geq4$, the result follows.
\end{enumerate}

By above, we have shown that $D_{k,n}-F$ is connected, which implies that $\lambda_2(D_{k,n})\geq2n+2k-4$. Thus, the lemma follows.
\end{proof}

However, $D_{2,n}$ is not super-$\lambda_2$ when $n\geq3$ since the edges coming from a complete subgraph $K_n$, namely $D_{0,n}$, form a non-trivial minimum restricted edge-cut.

\begin{lemma}\label{D-2-2-super-lambda-dash} $D_{2,2}$ is super-$\lambda_2$.
\end{lemma}
\begin{proof} Let $F$ be any edge subset of $D_{2,2}$ with $|F|=4$. We shall show that if $D_{2,2}-F$ contains no isolated vertex, then either $D_{2,2}-F$ is connected or $F$ isolates an edge of $D_{2,2}$. Notice that $D_{2,2}$ is constructed from seven disjoint 6-cycles ($D_{1,2}^i$), for convenience, denoted by $C_i$, $0\leq i\leq 6$. We may assume that $f^0$ is the largest one among $f^i$. We consider the following cases.

\begin{enumerate}[{Case }1:]
\item $f^0=1$. It is obvious that each $C_i-F^i$ is connected. By a similar argument of Case 1 in Theorem \ref{super-lambda}, it can be shown that $D_{2,2}-F$ is connected.

\item $f^0=2$. If $f^j\leq1$ for each $j\in\{0,\cdots,6\}\setminus\{i\}$, then $D_{2,2}-V(C_0)-F$ is connected. If $C_0-F^0$ contains a singleton $u$, then $u$ must be connected to $D_{2,2}-V(C_0)-F$. So we assume that $C_0-F^0$ contains an isolated edge $xy$. Furthermore, if one of the level 2 edges of $x$ and $y$ is not in $F$, then $xy$ is connected to $D_{2,2}-V(C_0)-F$; otherwise, $F$ isolates $xy$ in $D_{2,2}-F$. For each component of $C_0-F^0$ containing at least three vertices, clearly, it is connected to $D_{2,2}-V(C_0)-F$. So we assume that $f^j=2$ for some $j\in\{0,\cdots,6\}\setminus\{i\}$, say $j=1$. Since $C_0$ and $C_1$ are both 6-cycles, $C_0-F^0$ and $C_1-F^1$ have at most two components, respectively. Clearly, $F=F^0\cup F^1$ and $D_{2,2}-V(C_0)\cup V(C_1)-F$ is connected. If $u$ and $v$ are two singletons of $C_0-F^0$ and $C_1-F^1$, respectively, and $uv\in E(D_{2,2})$, then $F$ isolates $uv$ in $D_{2,2}-F$; otherwise, $D_{2,2}-F$ is connected.

\item $f^0\geq3$. Clearly, $D_{2,2}-V(C_0)-F$ is connected. If $C_0-F^0$ contains a singleton $u$, then $u$ is connected to $D_{2,2}-V(C_0)-F$ since $D_{2,2}-F$ contains no isolated vertex. If $C$ is a component of $C_0-F^0$ with $|C|\geq2$, then $C$ is connected to $D_{2,2}-V(C_0)-F$. Thus, $D_{2,2}-F$ is connected.
\end{enumerate}

Hence, the lemma follows.
\end{proof}

\begin{theorem}\label{D-k-n-super-lambda-dash} $D_{k,n}$ is super-$\lambda_2$ for $k\geq3$ and $n\geq2$, or $k=2$ and $n=2$.
\end{theorem}
\begin{proof} Let $F$ be any edge subset of $D_{k,n}$ with $|F|=2n+2k-4$. We keep the notation introduced in Theorem \ref{D-k-n-lambda-dash}. By Lemma \ref{D-2-2-super-lambda-dash}, it suffices to consider $k\geq3$ and $n\geq2$. We shall show that if $D_{k,n}-F$ contains no isolated vertex, then either $D_{k,n}-F$ is connected or $F$ isolates an edge of $D_{k,n}$. If each of $D_{k-1,n}^i-F^i$ is connected for $i\in\{0,1,\cdots,p\}$, then $D_{k,n}-F$ is connected. So we assume that one of $D_{k-1,n}^i-F^i$ is disconnected, say $D_{k-1,n}^0-F^0$. We consider the following cases.

\begin{enumerate}[{Case }1:]
\item $f^0=n+k-2$. Clearly, $\sum_{j=1}^p f^j\leq n+k-2$.  Furthermore, if each of $D_{k-1,n}^j-F^j$, $j\in\{1,\cdots,p\}$, is connected, then $D_{k,n}-V(D_{k-1,n}^0)-F$ is connected. By Theorem \ref{super-lambda}, $F^0$ isolates a singleton $u$ of $D_{k-1,n}^0$. Since there exists no isolated vertex in $D_{k,n}-F$, $u$ must connect to a vertex in $D_{k,n}-V(D_{k-1,n}^0)-F$. It is not hard to see that there exists a vertex of the larger part of $D_{k-1,n}^0-F^0$ connecting to a vertex in $D_{k,n}-V(D_{k-1,n}^0)-F$. Thus, $D_{k,n}-F$ is connected.

Now we may assume that $D_{k-1,n}^1-F^1$ is disconnected. At this time, $D_{k,n}-V(D_{k-1,n}^0)\cup V(D_{k-1,n}^1)-F$ is connected. Again, $F^0$ (resp. $F^1$) isolates a singleton $u$ (resp. $v$) of $D_{k-1,n}^0$ (resp. $D_{k-1,n}^1$). If $uv\in E(D_{k,n})$, then $F$ isolates an edge in $D_{k,n}$; otherwise, $D_{k,n}-F$ is connected.

\item $n+k-1\leq f^0\leq 2n+2k-7$. Clearly, $\sum_{i=1}^p f^i\leq n+k-3$ when $n+k\geq6$. Then $D_{k,n}-V(D_{k-1,n}^0)-F$ is connected. By Theorems \ref{super-lambda} and \ref{D-k-n-lambda-dash}, $F^0$ isolates exactly one singleton $u$ of $D_{k-1,n}^0$ when $n+k\geq6$. Hence $D_{k-1,n}^0-F^0$ contains two components $C_1$ and $C_2$, where $C_1$ is the singleton $u$. Since there is no isolated vertex in $D_{k,n}-F$, $u$ is connected to a vertex in $D_{k,n}-V(D_{k-1,n}^0)-F$. Obviously, there exists a vertex of $C_2$ connecting to a vertex in $D_{k,n}-V(D_{k-1,n}^0)-F$. Therefore, $D_{k,n}-F$ is connected.

Note that $2n+2k-7=n+k-2=3$ when $k=3$ and $n=2$, the proof is analogous to that of Case 1. Note also that $n+k-1=4$ when $k=3$ and $n=2$, $F^0$ may isolate a singleton or an isolated edge of $D_{2,2}^0$ since $\lambda_2(D_{2,2})=4$. As mentioned earlier, we only consider that $F^0$ isolates an isolated edge, say $xy$, of $D_{2,2}^0$. If $xx^k\in F$ and $yy^k\in F$ hold, then $F$ isolates an isolated edge of $D_{3,2}$; otherwise, $D_{k,n}-F$ is connected.

\item $f^0\geq 2n+2k-6$. It suffices to consider $D_{k,n}$ with $n+k\geq6$ since $2n+2k-6=4$ when $k=3$ and $n=2$. If $u$ is an isolated vertex in $D_{k-1,n}^0-F^0$, then the level $k$ edge $uu^k\not\in F$, which implies that $u$ is connected to a vertex in $D_{k,n}-V(D_{k-1,n}^0)-F$. So we assume that $uv$ is an isolated edge in $D_{k-1,n}^0-F^0$. If the level $k$ edges $uu^k\not\in F$ or $vv^k\not\in F$, then $u$ or $v$ is connected to a vertex in $D_{k,n}-V(D_{k-1,n}^0)-F$; otherwise, $F$ isolates an edge of $D_{k,n}$. For any component $C$ of $D_{k-1,n}^0-F^0$ with $|V(C)|\geq3$, noting that at most two edges are deleted outside $D_{k-1,n}^0$, each vertex in $C$ has a neighbor in $D_{k,n}-V(D_{k-1,n}^0)-F$. It implies that $D_{k,n}-F$ is connected. Thus, the theorem follows.
\end{enumerate}
\end{proof}

In what follows, we shall consider 3-restricted edge-connectivity of $D_{k,n}$. The following lemma is needed.

\begin{lemma}\label{Complete-super-lambda-3}\cite{Balbuena} The complete graph $K_n$ is super-$\lambda_3$ for $n\geq 6$.
\end{lemma}

\begin{theorem}\label{D-k-n-lambda-3} $\lambda_3(D_{k,n})=3n+3k-9$ for all $n\geq3$ and $k\geq3$.
\end{theorem}
\begin{proof} Pick out a path $P$ of length two or a triangle $C$ of $D_{k,n}$ for $n\geq3$ and $k\geq3$. Clearly, $\lambda_3(D_{k,n})\leq\min\{|[V(P),\overline{V(P)}]|,|[V(C),\overline{V(C)}]|\}=3n+3k-9$. It suffices to prove that $\lambda_3(D_{k,n})\geq 3n+3k-9$.

Let $F\subset E(D_{k,n})$ with $|F|=3n+3k-10$ such that there are neither isolated vertices nor
isolated edges in $D_{k,n}-F$. Our objective is to show that $D_{k,n}-F$ is connected. Observe that $3(n+k-2)>3n+3k-10$ and $D_{k-1,n}^i$ is ($n+k-2$) edge-connected, then at most two of $D_{k-1,n}^i-F^i$, $0\leq i\leq p$, are disconnected. In fact, $2n+2k-4=3n+3k-10=8$ when $n=3$ and $k=3$, by Theorem \ref{D-k-n-super-lambda-dash}, it implies that $D_{k,n}-F$ is connected. So we assume that $n\geq3$ and $k\geq4$, or $n\geq4$ and $k\geq3$. We consider the following cases.

\begin{enumerate}[{Case }1:]
\item For each $0\leq i\leq p$, $D_{k-1,n}^i-F^i$ is connected. Since $p=|D_{k-1,n}|\geq n(n+1)\cdots(n+k-1)$, we have $p>3n+3k-10$ whenever $n\geq3$ and $k\geq3$, by the proof of the Case 1 of Theorem \ref{super-lambda}, $D_{k,n}-F$ is connected.

\item Exactly one of $D_{k-1,n}^i-F^i$ is disconnected. We may assume that $D_{k-1,n}^0-F^0$ is disconnected. Then $f^0\geq n+k-2$. Since each of $D_{k-1,n}^i-F^i$ ($i\neq0$) is connected, we can obtain that $D_{k,n}-V(D_{k-1,n}^0)-F$ is connected. We need the following claim.

\noindent{\bf Claim.} Each vertex in $D_{k-1,n}^0-F^0$ is connected to a vertex in $D_{k,n}-V(D_{k-1,n}^0)-F$ via a fault-free path in $D_{k,n}-F$.

\noindent{\bf Proof of the Claim:} Let $u$ be an arbitrary vertex in $D_{k-1,n}^0-F^0$. If $e_k(u)\not\in F$, we are done. So we assume that $e_k(u)\in F$. Since there are no isolated vertices in $D_{k-1,n}^0-F^0$, there is an edge $uv\in E(D_{k-1,n}^0)$ such that $uv\not\in F$. If $e_k(v)\not\in F$, we are done. Similarly, we assume that $e_k(v)\in F$. Moreover, there are no isolated edges in $D_{k-1,n}^0-F^0$, then there is an edge $uw$ or $vw$, say $vw$, in $D_{k-1,n}^0-F^0$. Again, if $e_k(w)\not\in F$, we are done. So we assume that $e_k(w)\in F$. We consider the following three conditions.

(1) Both of $uv$ and $vw$ are level 0 edges. That is, $u,v$ and $w$ are vertices of some $K_n$ in $D_{k-1,n}^0$. In addition, $u$ (resp. $v$, $w$) has $k-1$ distinct neighbors $u_j$ (resp. $v_j$, $w_j$) in $D_{k-1,n}^0$ but outside $K_n$, $1\leq j\leq k-1$. So there exist $k-1$ edge disjoint paths $P_j=uu_ju_j^k$ (resp. $Q_j=vv_jv_j^k$, $W_j=ww_jw_j^k$) from $u$ (resp. $v$, $w$) to $u_j^k$ (resp. $v_j^k$, $w_j^k$), where $u_j^k$ (resp. $v_j^k$, $w_j^k$) is a level $k$ neighbor of $u_j$ (resp. $v_j$, $w_j$). If at least one of $P_j$, $Q_j$ and $W_j$ is fault-free in $D_{k,n}-F$, we are done. So we assume that each of $P_j$, $Q_j$ and $W_j$ has at least one edge in $F$.

There are at most $3n+3k-10-(3\times(k-1)+3)=3n-10$ edges of $F$ in the $K_n$. Since $3n-10<0$ when $n=3$, we need only to consider $n\geq4$. In addition, by Lemmas \ref{Complete-super-lambda-dash} and \ref{Complete-super-lambda-3}, $K_n$ is super-$\lambda_2$ and super-$\lambda_3$ when $n\geq4$ and $n\geq6$, respectively. In other words, $\lambda_2(K_n)=2n-4$ when $n\geq4$ and $\lambda_3(K_n)=3n-9$ when $n\geq6$. If the $K_n-F$ containing $uv$ is connected, for each vertex $x_l$ of $K_n$ (not $u$, $v$ and $w$), $1\leq l\leq n-3$, there are $k$ edge disjoint paths from $x_l$ to a vertex in $D_{k,n}-V(D_{k-1,n}^0)-F$. Since $k(n-3)>3n-10$ whenever $n\geq4$ and $k\geq3$, there exists a fault-free path from $u$ to a vertex in $D_{k,n}-V(D_{k-1,n}^0)-F$, we are done. Now we assume that $K_n-F$ is disconnected. Since $n-1>3n-10$ when $n=4$, we only need to consider $n\geq5$.

There is exactly one singleton in $K_5-F$ since $3n-10<2n-4$ when $n=5$. It is not difficult to see the claim holds. So we assume that $n\geq6$. When at most $3n-10$ edges are deleted from $K_{n}$, the resulting graph is either connected, or contains exactly two components, one of which is a singleton or an edge, or contains exactly three components, two of which are singletons. Let the component of $K_n-F$ containing $u$, $v$ and $w$ be $C$, then $C$ contains at least $n-5$ vertices except $u$, $v$ and $w$. There are at least $n-1$ edges to separate $C$ from $K_n$. Note $k(n-5)\geq(3n-10)-(n-1)$ when $n\geq6$. In fact, $k(n-5)>(3n-10)-(n-1)$ when $n\geq7$. When $n=6$, if we take $n-1$ on the right side, the left side is $k(n-4)$. Therefore, there exists a fault-free path from $u$ to a vertex in $D_{k,n}-V(D_{k-1,n}^0)-F$.

(2) Either $uv$ or $vw$ is a level 0 edge, but not both. Without loss of generality, suppose that $uv$ is a level 0 edge. Similarly, $u$ has $k-1$ distinct neighbors $u_{j_1}$ in $D_{k-1,n}^0$ but outside $K_n$, and $v$ (resp. $w$) has $k-2$ distinct neighbors $v_{j_2}$ (resp. $w_{j_2}$) in $D_{k-1,n}^0$ but outside $K_n$, where $1\leq j_1\leq k-1$, $1\leq j_2\leq k-2$, $v_{j_2}\neq w$ and $w_{j_2}\neq v$. So there exist $k-1$ edge disjoint paths $P_{j_1}=uu_{j_1}u_{j_1}^k$ from $u$ to $u_{j_1}^k$, and there exist $k-2$ edge disjoint paths $Q_{j_2}=vv_{j_2}v_{j_2}^k$ (resp. $W_{j_2}=ww_{j_2}w_{j_2}^k$) from $v$ to $v_{j_2}^k$ (resp. $w$ to $w_{j_2}^k$). If at least one of $P_{j_1}$, $Q_{j_2}$ and $W_{j_2}$ is fault-free in $D_{k,n}-F$, we are done. So we assume that each of $P_{j_1}$, $Q_{j_2}$ and $W_{j_2}$ has at least one edge in $F$.

For convenience, we denote the $K_n$ containing $u$ and $v$ by $K_n^1$ and the $K_n$ containing $w$ by $K_n^2$, respectively. Thus, there are at most $3n+3k-10-(k-1+2\times(k-2)+3)=3n-8$ edges of $F$ in the $K_n^1$ and $K_n^2$. If one of $K_n^1-F$ and $K_n^2-F$, say $K_n^1-F$, is connected, then there are at least $k(n-2)$ edge disjoint paths from vertices of $K_n^1$ (except $u$ and $v$) to vertices in $D_{k,n}-V(D_{k-1,n}^0)-F$. Since $k(n-2)>3n-8$ whenever $n\geq3$, there exists a fault-free path from $u$ to a vertex in $D_{k,n}-V(D_{k-1,n}^0)-F$, we are done. So we assume that both of $K_n^1-F$ and $K_n^2-F$ are disconnected. Since $n-1+(2n-4)>3n-8$ when $n\geq3$, each of $K_n^1-F$ and $K_n^2-F$ has a singleton. Clearly, $|F\cap E(K_n^1)|\leq 3n-8-(n-1)=2n-7$. It suffices to consider $n\geq6$ since $2n-7<n-1$ when $n<6$. At this time, there are at least $k(n-3)$ edge disjoint paths from vertices of $K_n^1$ (except $u$, $v$ and the singleton in $K_n^1-F$) to vertices in $D_{k,n}-V(D_{k-1,n}^0)-F$. It is obvious that $k(n-3)>2n-7$ whenever $n\geq6$ and $k\geq3$, then there exists a fault-free path from $u$ to a vertex in $D_{k,n}-V(D_{k-1,n}^0)-F$.

(3) Neither $uv$ nor $vw$ is a level 0 edge. Noting that $D_{k-1,n}^0$ is $(n+k-2)$-regular, then $u$ (resp. $v$,$w$) has at least $n+k-4$ neighbors $u_l$ (resp. $v_l$,$w_l$), $1\leq l\leq n+k-4$, in $D_{k-1,n}^0$, where $u_l\neq v,w$, $v_l\neq u,w$ and $w_l\neq u,v$. Thus, there are $n+k-4$ edge-disjoint paths $uu_ju^k_j$ (resp. $vv_jv^k_j$, $ww_jw^k_j$) of length two from $u$ (resp. $v$,$w$) to $u^k_j$ (resp. $v^k_j$,$w^k_j$). There are at least $3(n+k-4)+3=3n+3k-9>3n+3k-10$ edge disjoint paths from $u$, $v$ and $w$ to $D_{k,n}-V(D_{k-1,n}^0)-F$, which implies that the claim holds. \qed

By the above claim, it follows that $D_{k,n}-F$ is connected.

\item Exactly two of $D_{k-1,n}^i-F^i$ are disconnected. We may assume that $D_{k-1,n}^0-F^0$ and $D_{k-1,n}^1-F^1$ are disconnected. Then $f^0\geq n+k-2$ and $f^1\geq n+k-2$. It follows that $\sum_{i=2}^pf^i\leq 3n+3k-10-2(n+k-2)=n+k-6$. Clearly, $D_{k,n}-V(D_{k-1,n}^0)\cup V(D_{k-1,n}^1)-F$ is connected. We may assume that $f^0\geq f^1$. Then $f^0\leq 3n+3k-10-(n+k-2)=2n+2k-8$. By Theorem \ref{D-k-n-lambda-dash}, $\lambda_2(D_{k-1,n}^i)=2n+2(k-1)-4=2n+2k-6>2n+2k-8$, then $D_{k-1,n}^0-F^0$ (resp. $D_{k-1,n}^1-F^1$) contains exactly one singleton $x$ (resp. $y$). Clearly, $D_{k,n}-\{x,y\}-F$ is connected. If $xy\in E(D_{k,n})$, then $D_{k,n}-F$ contains an isolated edge or a singleton, a contradiction. Thus, $xy\not\in E(D_{k,n})$, then $D_{k,n}-F$ is connected since there exist no isolated vertices in $D_{k,n}-F$. This completes the proof.
\end{enumerate}
\end{proof}

Similarly, $D_{3,n}$ is not super-$\lambda_3$ when $n\geq4$ since the edges coming from a subgraph $K_n$, namely $D_{0,n}$, form a non-trivial minimum restricted edge cut. We shall consider $D_{k,n}$ for all $k\geq4$ and $n\geq3$ and obtain the following result.

\begin{theorem}\label{D-k-n-super-lambda-3} $D_{k,n}$ is super-$\lambda_3$ for all $k\geq4$ and $n\geq3$.
\end{theorem}
\begin{proof} Let $F$ be any edge subset of $D_{k,n}$ with $|F|=3n+3k-9$. We keep the notation introduced in Theorem \ref{D-k-n-lambda-3}. We shall show that if $D_{k,n}-F$ contains neither isolated vertices nor isolated edges, then either $D_{k,n}-F$ is connected or $F$ isolates a triangle of $D_{k,n}$. Observe that $3(n+k-2)>3n+3k-9$ and $D_{k-1,n}^i$ is ($n+k-2$) edge-connected, then at most two of $D_{k-1,n}^i-F^i$, $0\leq i\leq p$, are disconnected. Suppose without loss of generality that $f^0$ is the largest one among $f^i$. If each $D_{k-1,n}^i-F^i$ is connected, then $D_{k,n}-F$ is connected. So we consider the following cases.

\begin{enumerate}[{Case }1:]
\item $n+k-2\leq f^0\leq 2n+2k-7$. We may assume that $D_{k-1,n}^0-F^0$ is disconnected. Furthermore, if each of $D_{k-1,n}^j-F^j$, $j\in\{1,\cdots,p\}$, is connected, then $D_{k,n}-V(D_{k-1,n}^0)-F$ is connected. By Theorem \ref{super-lambda}, $F^0$ isolates a singleton $u$ of $D_{k-1,n}^0$. Since there exists no isolated vertex in $D_{k,n}-F$, $u$ must connect to a vertex in $D_{k,n}-V(D_{k-1,n}^0)-F$. It is not hard to see that there exists a vertex in $D_{k-1,n}^0-u-F^0$ connecting to a vertex in $D_{k,n}-V(D_{k-1,n}^0)-F$. Thus, $D_{k,n}-F$ is connected.

Now we assume that one of $D_{k-1,n}^j-F^j$, say $D_{k-1,n}^1-F^1$, is disconnected. At this time, $D_{k,n}-V(D_{k-1,n}^0)\cup V(D_{k-1,n}^1)-F$ is connected. Again, $F^0$ (resp. $F^1$) isolates a singleton $u$ (resp. $v$) of $D_{k-1,n}^0$ (resp. $D_{k-1,n}^1$). If $uv\in E(D_{k,n})$, then $F$ isolates an edge in $D_{k,n}$, a contradiction; otherwise, $D_{k,n}-F$ is connected.

\item $2n+2k-6\leq f^0\leq 3n+3k-13$. Clearly, $D_{k,n}-V(D_{k-1,n}^0)-F$ is connected. By Theorems \ref{super-lambda} and \ref{D-k-n-super-lambda-dash}, it can be shown that $D_{k-1,n}^0-F^0$ consists of at most two singletons $u$ and $v$, or exactly one isolated edge $uv$. If $D_{k-1,n}^0-F^0$ consists of a singleton $u$. Since there exists no isolated vertex in $D_{k,n}-F$, $u$ must connect to a vertex in $D_{k,n}-V(D_{k-1,n}^0)-F$. If $D_{k-1,n}^0-F^0$ consists of exactly one isolated edge $uv$, then $uu^k\not\in F$ or $vv^k\not\in F$, indicating that $D_{k,n}-F$ is connected.

\item $f^0\geq 3n+3k-12$. Obviously, $D_{k,n}-V(D_{k-1,n}^0)-F$ is connected and $|F\cap E_k|\leq3$. If $D_{k-1,n}^0-F^0$ contains a singleton or an isolated edge, by our assumption, then it must connect to a vertex in $D_{k,n}-V(D_{k-1,n}^0)-F$. So we assume that each component $C$ of $D_{k-1,n}^0-F^0$ has order at least three. If $|C|=3$ and there exists a vertex $u\in V(C)$ such that $uu^k\not\in F$, then $C$ is connected to $D_{k,n}-V(D_{k-1,n}^0)-F$, we are done. Suppose not. Then $C$ is a component of $D_{k,n}-F$. Meanwhile $f^0=3n+3k-12$. This implies that $F$ isolates a triangle $C$. If $|C|>3$, then $C$ is obviously connected to $D_{k,n}-V(D_{k-1,n}^0)-F$, indicating that $D_{k,n}-F$ is connected.
\end{enumerate}
\end{proof}

Observe that $D_{k,2}$ is triangle-free, we consider its 3-restricted edge-connectivity as follows.

\begin{theorem}\label{D-k-2-lambda-3} $\lambda_3(D_{k,2})=3n+3k-7$ for all $k\geq2$.
\end{theorem}
\begin{proof}We pick out a path $P$ of length two of $D_{k,2}$ for $k\geq2$. Clearly, $\lambda_3(D_{k,2})\leq\min\{|[V(P),\overline{V(P)}]|\}=3n+3k-7$. It suffices to prove that $\lambda_3(D_{k,2})\geq 3n+3k-7$.

Let $F\subset E(D_{k,2})$ with $|F|=3n+3k-8$ such that there are neither isolated vertices nor
isolated edges in $D_{k,2}-F$. Our aim is to show that $D_{k,2}-F$ is connected. Observe that $3(n+k-2)>3n+3k-8$ and $D_{k-1,n}^i$ is ($n+k-2$) edge-connected, then at most two of $D_{k-1,n}^i-F^i$, $0\leq i\leq p$, are disconnected. In fact, $2n+2k-4=3n+3k-8=4$ when $k=2$ and $n=2$, by Lemma \ref{D-2-2-super-lambda-dash}, which implies that $D_{2,2}-F$ is connected. So we assume that $k\geq3$. Since $p=|D_{k-1,n}|\geq n(n+1)\cdots(n+k-1)$, we have $p>3n+3k-8$ whenever $n=2$ and $k\geq3$. For each $0\leq i\leq p$, if $D_{k-1,n}^i-F^i$ is connected, by the proof of Case 1 of Theorem \ref{super-lambda}, then $D_{k,n}-F$ is connected. We assume that there exists some $i$, $0\leq i\leq p$, such that $D_{k-1,n}^i-F^i$ is disconnected. We consider the following cases.

\begin{enumerate}[{Case }1:]
\item Exactly one of $D_{k-1,n}^i-F^i$ is disconnected. We may assume that $D_{k-1,n}^0-F^0$ is disconnected. Then $f^0\geq n+k-2$. Since each of $D_{k-1,n}^i-F^i$ ($i\neq0$) is connected, we can obtain that $D_{k,n}-V(D_{k-1,n}^0)-F$ is connected. We claim that each vertex in $D_{k-1,n}^0-F^0$ is connected to a vertex in $D_{k,n}-V(D_{k-1,n}^0)-F$ via a path in $D_{k,n}-F$.

Let $u$ be an arbitrary vertex in $D_{k-1,n}^0-F^0$. If $e_k(u)\not\in F$, we are done. So we assume that $e_k(u)\in F$. Since there are no isolated vertices in $D_{k-1,n}^0-F^0$, there is an edge $uv\in E(D_{k-1,n}^0)$ such that $uv\not\in F$. If $e_k(v)\not\in F$, we are done. Similarly, we assume that $e_k(v)\in F$. Moreover, there are no isolated edges in $D_{k-1,n}^0-F^0$. Then there is an edge $uw$ or $vw$, say $vw$, in $D_{k-1,n}^0-F^0$. Again, if $e_k(w)\not\in F$, we are done. So we assume that $e_k(w)\in F$.

Noting that $D_{k-1,n}^0$ is $(n+k-2)$-regular, then $u$ (resp. $w$) has $n+k-3$ neighbors $u_{l_1}$ (resp. $w_{l_1}$), $1\leq l_1\leq n+k-3$, in $D_{k-1,n}^0$, where $u_{l_1}\neq v,w$ and $w_{l_1}\neq u,v$. Similarly, $v$ has $n+k-4$ neighbors $v_{l_2}$, $1\leq {l_2}\leq n+k-4$, in $D_{k-1,n}^0$, where $v_{l_2}\neq v,w$.
Thus, there are $n+k-3$ edge-disjoint paths $uu_{l_1}u_{l_1}^k$ (resp. $ww_{l_1}w_{l_1}^k$) of length two from $u$ (resp. $w$) to $u_{l_1}^k$ (resp. $w_{l_1}^k$) and $n+k-4$ edge-disjoint paths $vv_{l_2}v_{l_2}^k$ of length two from $v$ to $v_{l_2}^k$. There are $2(n+k-3)+n+k-4+3=3n+3k-7>3n+3k-8$ edge disjoint paths from $u$, $v$ and $w$ to $D_{k,n}-V(D_{k-1,n}^0)-F$ in total, which implies that the claim holds. Thus, $D_{k,n}-F$ is connected.
\item Exactly two of $D_{k-1,n}^i-F^i$ are disconnected. We may assume that $D_{k-1,n}^0-F^0$ and $D_{k-1,n}^1-F^1$ are disconnected. Then $f^0\geq n+k-2$ and $f^1\geq n+k-2$. It follows that $\sum_{i=2}^pf^i\leq 3n+3k-8-2(n+k-2)=n+k-4$. Clearly, $D_{k,n}-V(D_{k-1,n}^0)\cup V(D_{k-1,n}^1)-F$ is connected. We may assume that $f^0\geq f^1$. Then $f^0\leq 3n+3k-8-(n+k-2)=2n+2k-6$. By Theorem \ref{D-k-n-lambda-dash}, $\lambda_2(D_{k-1,n}^i)=2n+2(k-1)-4=2n+2k-6$, then $D_{k-1,n}^0-F^0$ contains exactly one singleton $x$ or one isolated edge $xy$ and $D_{k-1,n}^1-F^1$ contains exactly one singleton $z$. If $D_{k-1,n}^0-F^0$ contains exactly one isolated edge $xy$, then $F\cap E_k=\emptyset$. Obviously, $D_{k,n}-F$ is connected. We assume that $D_{k-1,n}^0-F^0$ contains exactly one singleton $x$ and $D_{k-1,n}^1-F^1$ contains exactly one singleton $z$. By our assumption, $e_k(x)\not\in F$ and $e_k(z)\not\in F$. If $x^k=z$, then $xz\in E(D_{k,n})$, which implies that $xz$ is an isolated edge in $D_{k,n}-F$, a contradiction; otherwise, $x^k\neq z$, then it is not difficult to see that $D_{k,n}-F$ is connected.
\end{enumerate}
This completes the proof.
\end{proof}

\begin{lemma}\label{D-2-2-super-lambda-3} $D_{2,2}$ is super-$\lambda_3$.
\end{lemma}
\begin{proof} By Theorem \ref{D-k-2-lambda-3}, let $F$ be any edge subset of $D_{2,2}$ with $|F|=5$. We shall show that if $D_{2,2}-F$ contains no singleton and no isolated edge, then either $D_{2,2}-F$ is connected or $F$ isolates a path of length two of $D_{2,2}$. For simplicity, we denote each $D_{1,2}^i$ by $C_i$, $0\leq i\leq 6$. We may assume that $f^0$ is the largest one among $f^i$. If each $C_i-F^i$ is connected, it can be shown that $D_{2,2}-F$ is connected. So we assume that $C_0-F^0$ is disconnected.

\begin{enumerate}[{Case }1:]
\item $2\leq f^0\leq3$. If $C_j-F^j$ is connected for every $j\geq1$, then $D_{2,2}-V(C_0)-F$ is connected. If $C_0-F^0$ contains a singleton $u$ (resp. an edge $xy$), then $u$ (resp. $xy$) must be connected to $D_{2,2}-V(C_0)-F$. So we assume that $C$ is a component of $C_0-F^0$ with $|V(C)|\geq3$. If $|V(C)|=3$ and exactly three level 2 edges incident to $C$ are contained in $F$, then $F$ isolates a path of length 2 in $D_{2,2}-F$ (this implies that $f^0=2$); otherwise, $D_{2,2}-F$ is connected. Obviously, a vertex of $C$ is connected to a vertex in $D_{2,2}-V(C_0)-F$ when $|V(C)|\geq4$. Thus, $D_{2,2}-F$ is connected.

We assume that exactly one of $C_j-F^j$, say $C_1-F^1$ is disconnected for some $j\in\{1,\cdots,6\}$. Clearly, $D_{2,2}-V(C_0)\cup V(C_1)-F$ is connected. It follows that $|F\cap E_2|\leq 1$. This implies that for any component of $C_0-F^0$ and $C_1-F^1$ having at least two vertices, there exists a vertex connecting to $D_{2,2}-V(C_0)\cup V(C_1)-F$. So we only consider the singletons in $C_0-F^0$ and $C_1-F^1$. Let $u$ and $v$ be two singletons of $C_0-F^0$ and $C_1-F^1$, respectively. If $uv\in E(D_{2,2})$, then $F$ isolates $uv$ in $D_{2,2}-F$, a contradiction; otherwise, $D_{2,2}-F$ is connected.
\item $4\leq f^0\leq5$. Clearly, $D_{2,2}-V(C_0)-F$ is connected. If $C_0-F^0$ contains a singleton $u$, then $u$ is connected to $D_{2,2}-V(C_0)-F$ since $D_{2,2}-F$ contains no singleton. If $C$ is a component of $C_0-F^0$ with $|V(C)|\geq2$, then $C$ is connected to $D_{2,2}-V(C_0)-F$ since $|F\cap E_2|\leq1$. Thus, $D_{2,2}-F$ is connected.
\end{enumerate}
\end{proof}

\begin{theorem}\label{D-k-2-super-lambda-3} $D_{k,2}$ is super-$\lambda_3$ for all $k\geq2$.
\end{theorem}
\begin{proof} By Lemma \ref{D-2-2-super-lambda-3}, it remains to consider $k\geq3$. Let $F$ be any edge subset of $D_{k,2}$ with $|F|=3n+3k-7$. We keep the notation introduced in Theorem \ref{D-k-n-lambda-3}. We shall show that if $D_{k,2}-F$ contains neither isolated vertices nor isolated edges, then either $D_{k,2}-F$ is connected or $F$ isolates a path of length two of $D_{k,2}$. Observe that $3(n+k-2)>3n+3k-7$ and $D_{k-1,2}^i$ is ($n+k-2$) edge-connected, then at most two of $D_{k-1,2}^i-F^i$, $0\leq i\leq p$, are disconnected. Suppose without loss of generality that $f^0$ is the largest one among $f^i$. If each $D_{k-1,2}^i-F^i$ is connected, by Theorem \ref{D-k-n-super-lambda-dash}, then $D_{k,2}-F$ is connected. So we may assume that $D_{k-1,2}^0-F^0$ is disconnected. We consider the following cases.

\begin{enumerate}[{Case }1:]
\item $n+k-2\leq f^0\leq 2n+2k-7$. If $D_{k-1,2}^j-F^j$ is connected for each $j\in\{1,\cdots,p\}$, then $D_{k,2}-V(D_{k-1,2}^0)-F$ is connected. By Theorems \ref{super-lambda} and \ref{D-k-n-super-lambda-dash}, then $F^0$ isolates a singleton $u$ of $D_{k-1,2}^0$. Since there exists no isolated vertex in $D_{k,2}-F$, $u$ must connect to a vertex in $D_{k,2}-V(D_{k-1,2}^0)-F$. Obviously, there exists a vertex in $D_{k-1,2}^0-u-F^0$ connecting to a vertex in $D_{k,2}-V(D_{k-1,2}^0)-F$. Thus, $D_{k,2}-F$ is connected.

Now we assume that for some $j\in\{1,\cdots,p\}$, $D_{k-1,2}^j-F^j$, say $D_{k-1,2}^1-F^1$ is disconnected. At this time, $D_{k,2}-V(D_{k-1,2}^0)\cup V(D_{k-1,2}^1)-F$ is connected. Again, $F^0$ (resp. $F^1$) isolates a singleton $u$ (resp. $v$) of $D_{k-1,2}^0$ (resp. $D_{k-1,2}^1$) since $f^0\geq f^1$ and $\lambda_2(D_{k-1,2})=2n+2k-6$. If $uv\in E(D_{k,2})$, then $F$ isolates an edge in $D_{k,2}$, a contradiction; otherwise, $D_{k,2}-F$ is connected.
\item $2n+2k-6\leq f^0\leq3n+3k-11$. Similarly, if each of $D_{k-1,2}^j-F^j$ is connected for $j\in\{1,\cdots,p\}$, then $D_{k,2}-V(D_{k-1,2}^0)-F$ is connected. By Theorems \ref{super-lambda} and \ref{D-k-n-super-lambda-dash}, it can be shown that $D_{k-1,2}^0-F^0$ consists of at most two singletons $u$ and $v$, or exactly one isolated edge $uv$. If $D_{k-1,2}^0-F^0$ consists of a singleton $u$, then $u$ must connect to a vertex in $D_{k,2}-V(D_{k-1,2}^0)-F$ since there exists no isolated vertex in $D_{k,2}-F$. If $D_{k-1,2}^0-F^0$ consists of exactly one isolated edge $uv$, then $uu^k\not\in F$ or $vv^k\not\in F$, indicating that $D_{k,2}-F$ is connected.

Now we assume that for some $j\in\{1,\cdots,p\}$, $D_{k-1,2}^j-F^j$, say $D_{k-1,2}^1-F^1$ is disconnected. At this time, $D_{k,2}-V(D_{k-1,2}^0)\cup V(D_{k-1,2}^1)-F$ is connected. Since $2(2n+2k-6)=3n+3k-7=8$ when $k=3$, each of $D_{k-1,2}^0-F^0$ and $D_{k-1,2}^1-F^1$ may contain an isolated edge. Obviously, $D_{3,2}-F$ is connected in this case. So we assume that $F^0$ isolates a vertex $x$ or an edge $xy$ of $D_{k-1,2}^0$ and $F^1$ isolates a singleton $w$ of $D_{k-1,2}^1$. It is easy to know that $D_{k,n}-F$ is connected when $F^0$ isolates a vertex $x$ of $D_{k-1,2}^0$. Therefore, we assume that $F^0$ isolates an edge $xy$ of $D_{k-1,2}^0$. Furthermore, if $xw$ or $yw$, say $yw$, is an edge in $D_{k,2}-F$ and $xx^k\in F$, then $F$ isolates a path $xyw$ in $D_{k,2}$. Otherwise, $D_{k,2}-F$ is connected since $|F\cap E_k|\leq 1$.
\item $f^0\geq3n+3k-10$. Note that for some $j\in\{1,\cdots,p\}$, $D_{k-1,2}^j-F^j$ may be disconnected when $k=3$. Obviously, $D_{k,2}-V(D_{k-1,2}^0)-F$ is connected when $k\geq4$. If $D_{k-1,2}^0-F^0$ contains a singleton or an isolated edge, by our assumption, then it must connect to a vertex in $D_{k,2}-V(D_{k-1,2}^0)-F$. So we assume that each component $C$ of $D_{k-1,2}^0-F^0$ has order at least three. If $|V(C)|=3$ and there exists a vertex $u\in V(C)$ such that $uu^k\not\in F$, then $C$ is connected to $D_{k,2}-V(D_{k-1,2}^0)-F$, we are done. Suppose not. Then $C$ is a component of $D_{k,2}-F$. Meanwhile $f^0=3n+3k-10$. This implies that $F$ isolates a path of length two. If $|V(C)|>3$, then $C$ is obviously connected to $D_{k-1,2}^0-F^0$, indicating that $D_{k,2}-F$ is connected.

Now we consider $D_{3,2}$. We may assume that $D_{2,2}^1-F^1$ is disconnected. Obviously, $f^0=5$ and $f^1=3$. So $D_{2,2}^0-F^0$ may contain a path of length two, an isolated edge, or at most two singletons. Since $F\cap E_2=\emptyset$, each component of $D_{2,2}^0-F^0$ with at least two vertices is connected to $D_{3,2}-V(D_{2,2}^0)\cup V(D_{2,2}^1)-F$. Therefore, we may assume that $D_{2,2}^0-F^0$ contains two singletons $x$ and $y$, and $D_{2,2}^1-F^1$ contains a singleton $w$. If $xw\not\in E(D_{3,2})$ and $yw\not\in E(D_{3,2})$, then $x$, $y$ and $w$ is connected to $D_{3,2}-V(D_{2,2}^0)\cup V(D_{2,2}^1)-F$ by our assumption. If $xw\in E(D_{3,2})$ or $yw\in E(D_{3,2})$, then $F$ isolates an isolated edge, a contradiction.
\end{enumerate}
Hence, the theorem holds.
\end{proof}

\section{Matching preclusion and conditional matching preclusion of DCell}
\label{sec:matching}

We begin with some useful statements.

\begin{theorem}\label{PM}\cite{Plesnik} Let $G$ be an $r$-regular graph of even order. If $G$ is $(r-1)$-edge-connected, then $G-F$ has a perfect matching for every $F\subset E(G)$ with $|F|\leq r-1$.
\end{theorem}

This theorem obviously indicates that $mp(G)=r$ for an $r$-regular $(r-1)$-edge-connected graph $G$. The notation $\alpha(G)$ in theorems below is the independence number of $G$.

\begin{theorem}\label{super-matched}\cite{Cheng} Let $G$ be an $r$-regular graph of even order ($r\geq3$). Suppose $G$ is super-$\lambda$ and $\alpha(G)<\frac{|V|-2}{2}$. Then $G$ is super matched.
\end{theorem}

\begin{theorem}\label{mp1}\cite{Cheng2} Let $G$ be an $r$-regular graph of even order ($r\geq4$). Suppose $G$ contains a 3-cycle, $G$ is $r$-edge-connected and $G$ is super $(3r-8)$-edge-connected of order 2. Moreover, assume that $\alpha(G)<\frac{|V|-2}{2}-(2r-8)$. Then mp$_1(G)=2r-3$.
\end{theorem}

\begin{theorem}\label{conditional-super-mp1}\cite{Cheng2} Let $G$ be an $r$-regular graph of even order ($r\geq4$). Suppose $G$ contains a 3-cycle, mp$_1(G)=2r-3$, $|V(G)|\geq8$, $G$ is super-$\lambda$ and $G$ is super $(3r-6)$-edge-connected of order 3. Moreover, assume that $\alpha(G)<\min\{\frac{|V(G)|-4}{2},\frac{|V(G)|-2}{2}-(2r-6)\}$. Then $G$ is conditionally super matched.
\end{theorem}

\begin{theorem}\label{mp1-triangle-free}\cite{Cheng} Let $G$ be an $r$-regular graph of even order ($r\geq3$). Suppose that $G$ is triangle-free, $G$ is $r$-edge-connected and $G$ is super $(3k-6)$-edge-connected of order 2. Moreover, assume that $\alpha(G)<\frac{|V|-2}{2}-(2r-6)$. Then mp$_1(G)=2r-2$.
\end{theorem}

\begin{theorem}\label{conditional-super-mp1-triangle-free}\cite{Cheng} Let $G$ be an $r$-regular graph of even order ($r\geq3$). Suppose that $G$ is triangle-free, mp$_1(G)=2r-2$, $G$ is super-$\lambda$ and $G$ is super $(3r-4)$-edge-connected of order 3. Moreover, assume that $\alpha(G)<\frac{|V|-2}{2}-(2r-4)$. Then $G$ is conditionally super matched.
\end{theorem}

We give some upper bounds of $\alpha(D_{k,n})$ as follows.

\begin{lemma}\label{independent-D-k-n}\
\begin{enumerate}
\item $\alpha(D_{2,2})=19$;
\item $\alpha(D_{k,2})\leq\frac{19}{42}|V(D_{k,2})|$ for $k\geq3$;
\item $\alpha(D_{k,n})\leq\frac{1}{n}|V(D_{k,n})|$ for $n\geq3$.
\end{enumerate}
\end{lemma}
\begin{proof}
\begin{enumerate}
\item We obtain this result directly by using Magma \cite{Bosma};
\item Observe that $D_{k,2}$ is recursively constructed from $D_{2,2}$, we can split $D_{k,2}$ into $D_{2,2}$s. Since $|V(D_{2,2})|=42$, the result follows easily;
\item For $n\geq3$, each vertex is contained in exactly one complete subgraph $K_n$ of $D_{k,n}$. It is obvious that $\alpha(K_n)=1$, then $\alpha(D_{k,n})\leq\frac{1}{n}|V(D_{k,n})|$.
\end{enumerate}
\end{proof}

To determine the matching preclusion number of $D_{1,n}$ ($n\geq2$), we firstly present the definition of the $(n',k')$-star graph. Star graphs are one of the most popular interconnection networks introduced by \cite{Chiang}. The $(n',k')$-star graph with $1\leq k'<n'$, which is a variant of the star graphs, is governed by the two parameters $n'$ and $k'$. The vertex set of $S_{n',k'}$ consists of all $k'$-permutations generated from the set $\{1,2,\cdots,n'\}$. Two vertices $a_1a_2\cdots a_{k'}$ and $b_1b_2\cdots b_{k'}$ are adjacent if one of the following holds:

\begin{enumerate}
\item There exists some $r\in\{2,\cdots,k'\}$ such that $a_1=b_r$, $a_r=b_1$ and $a_i=b_i$ for all $i\in\{1,2,\cdots,k'\}\setminus\{1,r\}$;
\item $a_1\neq b_1$ and $a_i=b_i$ for all $i\in\{2,\cdots,k'\}$.
\end{enumerate}

Clearly, $S_{n',1}$ is the complete graph $K_{n'}$. $S_{n',2}$ is constructed from $n'$ copies of $S_{n'-1,1}$, namely $K_{n'-1}$. Each vertex in a $K_{n'-1}$ has a unique neighbor outside this $K_{n'-1}$. By above, it seems like $S_{n+1,k+1}$ and $D_{k,n}$ have the same base structure and similarly recursive construction rule, one may ask whether $S_{n+1,k+1}$ is isomorphic to $D_{k,n}$. For $k=1$, we give the positive answer below.

\begin{lemma}\label{isomorphism} $D_{1,n}$ is isomorphic to $S_{n+1,2}$.
\end{lemma}
\begin{proof}
We need to show that there exists an automorphism that maps one vertex of $D_{1,n}$ to that of $S_{n+1,2}$. Let $a=(a_1,a_2)$ and $b=b_1b_2$ be any two vertices in $D_{1,n}$ and $S_{n+1,2}$, respectively. Note that $0\leq a_1\leq n$, $0\leq a_2\leq n-1$ and $1\leq b_1,b_2\leq n+1$, we define $P=\{1,2,\cdots,n+1\}$. In addition, let $\phi_i(P)$ be the $i$th smallest element in $P$, $1\leq i\leq n$. For example, $\phi_2(P)=2$. We define the bijection $\varphi$ as follows:

\begin{itemize}
\item $\varphi(a_1)=b_2$ if $b_2=a_1+1$, and
\item $\varphi(a_2)=b_1$ if $b_1=\phi_{a_2+1}(P\setminus\{b_2\})$.
\end{itemize}

It remains to show that $\varphi$ preserves adjacency. Suppose without loss of generality that $ac$ is an edge of $D_{1,n}$ and $c=(c_1,c_2)$. We consider the following two cases.

\begin{enumerate}[{Case }1:]
\item $ac$ is an edge in $K_n$. Then $c_1=a_1$ and $c_2\in\{0,\cdots,n-1\}\setminus\{a_2\}$. Thus, it can be easily verified that $\varphi(a)\varphi(c)$ is an edge in $S_{n+1,2}$.
\item $ac$ is a level 1 edge. We may assume that $a_1<c_1$, thus, $c=(a_2+1,a_1)$. So we have $\varphi(a)=(\phi_{a_2+1}(P\setminus\{a_1+1\}))(a_1+1)$ and $\varphi(c)=(\phi_{a_1+1}(P\setminus\{a_2+2\}))(a_2+2)$. Since $c_1=a_2+1$, noting that $a_1<c_1$, we have $a_1\leq a_2$. Thus, $\phi_{a_2+1}(P\setminus\{a_1+1\})=a_2+2$ and $\phi_{a_1+1}(P\setminus\{a_2+2\})=a_1+1$. Obviously, $\varphi(a)\varphi(c)$ is an edge in $S_{n+1,2}$.
\end{enumerate}
Thus, the lemma follows.
\end{proof}

By above, we know that $D_{1,n}$ is isomorphic to $S_{n+1,2}$. A {\em semi-trivial matching preclusion set} of $D_{1,n}$ (or $S_{n+1,2}$) is defined to be a set of edges with exactly one end in a unique complete subgraph of $D_{1,n}$ (or $S_{n+1,2}$).

\begin{lemma}\label{mp-s-n-2}\cite{Cheng2} Let $n\geq4$. Then mp$(S_{n,2})=n-1$. Moreover, if $n$ is odd, then $S_{n,2}$ is super matched; if $n$ is even, then the only optimal solutions are the trivial and semi-trivial matching preclusion sets.
\end{lemma}

It is obvious that $D_{1,2}$ is a 6-cycle, and it is not super matched. For $D_{1,n}$ with $n\geq3$, by Lemmas \ref{isomorphism} and \ref{mp-s-n-2}, the following result is straightforward.

\begin{theorem}\label{mp-d-1-n} Let $n\geq3$. Then mp$(D_{1,n})=n$. Moreover, if $n$ is even, then $D_{1,n}$ is super matched; if $n$ is odd, then the only optimal solutions are the trivial and semi-trivial matching preclusion sets.
\end{theorem}

\cite{Lu2} showed that $D_{k,n}$ is not vertex-transitive for all $k\geq2$ and $n\geq2$, while the ($n,k$)-star graphs are vertex-transitive for $1\leq k<n$. So $D_{k,n}$ is not isomorphic to $S_{n+1,k+1}$ for $k\geq2$. For $D_{k,n}$ with $k\geq2$ and $n\geq2$, we have the following theorem.

\begin{theorem}\label{mp-d-k-n} Let $k\geq2$ and $n\geq2$. Then mp$(D_{k,n})=n+k-1$. Moreover, $D_{k,n}$ is super matched.
\end{theorem}
\begin{proof} Clearly, $D_{k,n}$ is $r$-regular, where $r=n+k-1$. By Theorem \ref{super-lambda}, we know that $D_{k,n}$ is super-$\lambda$ for $k\geq2$ and $n\geq2$. Moreover, by Lemma \ref{independent-D-k-n}, $\alpha(D_{k,n})<\frac{|V(D_{k,n})|-2}{2}$ obviously holds. Thus, by Theorem \ref{super-matched}, the theorem is true.
\end{proof}

The conditional matching preclusion numbers and optimal conditional solutions of $D_{k,n}$ are studied in the following two theorems.

\begin{theorem}\label{mp1-d-k-n} Let $k\geq3$ and $n\geq3$. Then mp$_1$$(D_{k,n})=2n+2k-5$. Moreover, $D_{k,n}$ is conditionally super matched for $k\geq4$ and $n\geq3$.
\end{theorem}
\begin{proof}
Let $r=n+k-1$, clearly, $D_{k,n}$ is $r$-regular. By Theorem \ref{super-lambda}, $D_{k,n}$ is super-$\lambda$ when $k\geq2$ and $n\geq2$.

By Theorems \ref{D-k-n-super-lambda-dash} and \ref{D-k-n-lambda-3}, we know that $\lambda_2(D_{k,n})=2n+2k-4$ and $\lambda_3(D_{k,n})=3n+3k-9$ for $k\geq3$ and $n\geq3$. Since $3r-8<2r-2$ when $r=n+k-1=5$, $D_{k,n}-F$ is either connected or contains exactly two components, one of which is a singleton. Observe that $2r-2\leq 3r-8<3r-6$ when $r=n+k-1\geq6$, so $D_{k,n}-F$ is either connected, or contains exactly two components, one of which is a singleton or an edge, or contains exactly three components, two of which are singletons. This implies that $D_{k,n}$ is super $(3r-8)$-edge-connected of order 2. In addition, by Lemma \ref{independent-D-k-n}, $\alpha(D_{k,n})\leq\frac{1}{3}|V(D_{k,n})|<\min\{\frac{|V(D_{k,n})|-4}{2},\frac{|V(D_{k,n})|-2}{2}-(2r-6)\}$ when $n\geq3$. Then, by Theorem \ref{mp1}, $mp_1$$(D_{k,n})=2n+2k-5$ for $k\geq3$ and $n\geq3$.

By Theorem \ref{D-k-n-super-lambda-3}, we have that $D_{k,n}$ is super-$\lambda_3$ when $k\geq4$ and $n\geq3$.
By using a similar argument above, we know that $D_{k,n}$ is super $(3r-6)$-edge-connected of order 3 for $k\geq4$ and $n\geq3$. Then by Theorem \ref{conditional-super-mp1}, $D_{k,n}$ is conditionally super matched for $k\geq4$ and $n\geq3$.
\end{proof}

\begin{theorem}\label{mp1-d-k-2} Let $k\geq2$ and $n=2$. Then mp$_1$$(D_{k,n})=2n+2k-4$. Moreover, $D_{k,n}$ is conditionally super matched for $k\geq3$ and $n=2$.
\end{theorem}
\begin{proof} Let $r=n+k-1$, clearly, $D_{k,2}$ is $r$-regular. By Theorem \ref{super-lambda}, $D_{k,2}$ is super-$\lambda$ when $k\geq2$ and $n\geq2$, indicating that $D_{k,2}$ is $r$-edge-connected. By Theorems \ref{D-k-n-super-lambda-dash} and \ref{D-k-2-super-lambda-3}, we know that $\lambda_2(D_{k,2})=2n+2k-4$ and $\lambda_3(D_{k,2})=3n+3k-7$ for $k\geq2$. Since $3r-6<2r-2$ when $r=3$, $D_{k,2}-F$ is either connected or contains exactly two components, one of which is a singleton.

Observe that $2r-2\leq 3r-6<3r-4$ when $r\geq4$, so $D_{k,n}-F$ is either connected, or contains exactly two components, one of which is a singleton or an edge, or contains exactly three components, two of which are singletons. This implies that $D_{k,2}$ is super $(3r-6)$-edge-connected of order 2. In addition, by Lemma \ref{independent-D-k-n}, $\alpha(D_{k,2})\leq\frac{19}{42}|V(D_{k,2})|<\frac{|V(D_{k,2})|-2}{2}-(2r-6)$ holds when $k\geq2$. Then, by Theorem \ref{mp1-triangle-free}, $mp_1(D_{k,2})=2n+2k-4$ for $k\geq2$.

By Theorem \ref{D-k-2-super-lambda-3}, we have that $D_{k,2}$ is super-$\lambda_3$ for all $k\geq2$. By using a similar argument above, we know that $D_{k,2}$ is super $(3r-4)$-edge-connected of order 3 for $k\geq2$. Moreover, $\alpha(D_{k,2})<\frac{|V(D_{k,2})|-2}{2}-(2r-4)$ holds when $k\geq3$. Then, by Theorem \ref{conditional-super-mp1-triangle-free}, $D_{k,2}$ is conditionally super matched for $k\geq3$.
\end{proof}

\acknowledgements
\label{sec:ack}
The authors are grateful to the anonymous referees for their comments and constructive suggestions that greatly improved the original manuscript.

\nocite{*}
\bibliographystyle{abbrvnat}
\bibliography{sample-dmtcs}
\label{sec:biblio}

\end{document}